\documentclass[12pt]{amsart}
\usepackage[all]{xy}
\usepackage[mathscr]{euscript}
\usepackage{amsmath,amssymb,amsthm}
\usepackage{hyperref}
\hypersetup{colorlinks=true,urlcolor=blue,citecolor=blue,linkcolor=blue}
\usepackage{courier}
\usepackage{amsmath,amscd,amssymb,latexsym,longtable,diagram, picture}
\usepackage{graphicx}
\usepackage{array}
\usepackage{color}
\usepackage{enumerate}
\usepackage{nicefrac}
\usepackage{listings}
\usepackage{latexsym,bm,bbm,mathrsfs}
\usepackage{hyperref,graphicx, enumerate}
\usepackage{epsfig, xcolor, graphicx}
\usepackage[shortlabels]{enumitem}
\usepackage{indentfirst, setspace}
\usepackage{tikz-cd}

\allowdisplaybreaks
\usepackage{caption}
\usepackage{cancel}
\usepackage[normalem]{ulem}
\usetikzlibrary{calc,matrix,arrows,decorations.markings}

\newcommand{\bbz}{\mathbb{Z}}
\newcommand{\bbq}{\mathbb{Q}}

\newcommand{\bbc}{\mathbb{C}}

\newcommand{\GL}{\mathrm{GL}}

\newcommand{\Aut}{\mathrm{Aut}}
\newcommand{\Gal}{\mathrm{Gal}}

\newcommand{\mat}{\begin{pmatrix}}
\newcommand{\emat}{\end{pmatrix}}

\newcommand{\Tr}{\mathrm{Tr}}

\newtheorem{theorem}{Theorem}[section]

\newtheorem{lemma}[theorem]{Lemma}
\newtheorem{rmk}[theorem]{Remark}

\newtheorem{hypothesis}[theorem]{Hypothesis}
\newtheorem*{theorem*}{Theorem}
\newtheorem{conjecture}{Conjecture}

\author{Seokhyun Choi and Bo-Hae Im}

\address{
Dept. of Mathematical Sciences, KAIST,
291 Daehak-ro, Yuseong-gu,
Daejeon 34141, South Korea
}
\email{sh021217@kaist.ac.kr}

\address{
Dept. of Mathematical Sciences, KAIST,
291 Daehak-ro, Yuseong-gu,
Daejeon 34141, South Korea
}
\email{bhim@kaist.ac.kr}

\date{\today}
\thanks{Bo-Hae Im was supported by Basic Science Research Program through the National Research Foundation of Korea(NRF) grant funded by the Korea government(MSIT)(NRF-2023R1A2C1002385).}

\begin{document}

\thispagestyle{empty}
\title[Larsen's conjecture for elliptic curves over $\bbq$]
{Larsen's conjecture for elliptic curves over $\bbq$ with analytic rank at most~$1$}

\begin{abstract} 
    We prove Larsen's conjecture for elliptic curves over $\bbq$ with analytic rank at most~$1$. Specifically, let $E/\bbq$  be an elliptic curve over $\bbq$. If $E/\bbq$ has analytic rank at most~$1$, then we prove that for any topologically finitely generated subgroup $G$ of $\Gal(\overline{\bbq}/\bbq)$, the rank of $E$ over the fixed subfield $\overline{\bbq}^G$ of $\overline{\bbq}$ under $G$ is infinite.
\end{abstract}

\maketitle

\section{Introduction}

Elliptic curves are fundamental objects in number theory and algebraic geometry, playing a central role in many deep conjectures and theorems. In number theory, a key focus is on elliptic curves defined over number fields, where their arithmetic properties are of great interest. One of the foundational results in this setting is the Mordell-Weil theorem, which provides crucial structural information about the set of rational points on an elliptic curve.

Let $E/K$ be an elliptic curve over a number field $K$. The Mordell-Weil theorem says that the group of $K$-rational points $E(K)$ forms a finitely generated abelian group. Consequently, we can write 
\[ E(K) \cong E(K)_{\text{tors}}\otimes \bbz^r,\]
where  $r$ is a non-negative integer called the rank of $E$ over $K$, and  $E(K)_{\text{tors}}$ is the torsion subgroup, which is finite. The rank  plays a fundamental role in the arithmetic of elliptic curves. In particular, the Birch and Swinnerton-Dyer conjecture predicts a deep connection between the rank and the behavior of the Hasse–Weil $L$-function of
$E$ at $s=1$. Although the rank is an essential invariant, determining its exact value for a given elliptic curve is often a difficult problem. While there are effective methods to compute ranks in certain cases, many aspects remain mysterious, including whether there exists a uniform bound on ranks of elliptic curves over a fixed number field.

In contrast to the Mordell-Weil theorem, the structure of $E$ over an algebraic closure of $K$ behaves quite differently. Frey and Jarden \cite{FJ74} proved that the rank of $E$ over an algebraically closure $\overline{K}$ of $K$ is infinite, i.e.,
\[\dim E(\overline{K}) \otimes \bbq = \infty.\]
Furthermore, they proved a stronger result: for almost all $(\sigma_1,\ldots,\sigma_n) \in G_K^n$ in the sense of Haar measure, the rank of $E$ over the fixed subfield $\overline{K}^{\langle\sigma_1,\ldots,\sigma_n \rangle}$ of $\overline{K}$ under the group $\langle \sigma_1,\ldots,\sigma_n \rangle$, is infinite. 

In this vein, M. Larsen made the following conjecture:

\begin{conjecture}\label{conj:Larsen}{\normalfont (\cite{La03})}
    Let $E/K$ be an elliptic curve over a number field $K$. Let $G$ be a topologically finitely generated subgroup of $G_K:=\Gal(\overline{K}/K)$. Then the rank of $E$ over the fixed subfield $\overline{K}^G$ of $\overline{K}$ under $G$ is infinite.
\end{conjecture}

In a series of papers \cite{La03}, \cite{Im06}, \cite{Im06-2}, \cite{IL08}, the second author and Larsen used a geometric approach to prove Conjecture~\ref{conj:Larsen} when $G$ is procyclic, meaning it is topologically generated by a single element. Further, in \cite{Im07} and \cite{BI08}, Breuer and the second author applied the Heegner point method to prove Conjecture~\ref{conj:Larsen} in the case when $K$ is a totally real number field and $G$ is procyclic. In 2009, Tim and Vladimir Dokchitser \cite{DD09} proved Conjecture~\ref{conj:Larsen} when $K=\bbq$ under the assumption of either the rank part of Birch and Swinnerton-Dyer conjecture or the finiteness conjecture of the Tate-Shafarevich groups. In 2013, the second author and Larsen \cite{IL13} used combinatorial methods to prove Conjecture~\ref{conj:Larsen} when all $2$-torsion points of $E/K$ are $K$-rational. One can found more information about Larsen's conjecture in \cite{IL19}.

This paper is motivated by Conjecture~\ref{conj:Larsen}, and  we prove Conjecture~\ref{conj:Larsen} in the case $K=\bbq$ under the assumption that the given elliptic curve $E/\bbq$ has analytic rank at most~$1$: 

\begin{theorem}\label{main_theorem}
    Let $E/\bbq$ be an elliptic curve over $\bbq$. Suppose the analytic rank of $E/\bbq$ is at most $1$. Let $G$ be a topologically finitely generated subgroup of $G_\bbq := \Gal(\overline{\bbq}/\bbq)$. Then the rank of $E$ over $\overline{\bbq}^G$ is infinite.
\end{theorem}

\begin{rmk} 
   \
   
    \begin{enumerate}[\normalfont (a)]
        \item In 1979, Goldfeld conjectured that a family of quadratic twists  $\{E_D\}$ of an elliptic curve $E/\bbq$ over $\bbq$, $50\%$  of them have analytic rank~$0$, $50\%$  have analytic rank~$1$, and $0\%$  have analytic rank greater than or equal to~$2$. Our result, therefore, implies that under Goldfeld's conjecture, almost all quadratic twists of an elliptic curve over $\bbq$, in a probabilistic sense, satisfy Larsen's conjecture,~Conjecture~\ref{conj:Larsen}.
        \item By the results of Gross-Zagier and Kolyvagin (\cite[Theorem~3.22]{Da04}), if $E/\bbq$ is an elliptic curve over $\bbq$ with analytic rank at most~$ 1$, then the analytic rank coincides with the algebraic rank. Consequently, the assumption on the analytic rank in Theorem~\ref{main_theorem} implies that the algebraic rank is at most~$1$.
        \item The results of \cite{Im07} for $K=\bbq$ were later generalized in  \cite{BI08} to totally real number fields $K$, replacing modular curves by with Shimura varieties. In a similar vein, we expect that our results may also extend to totally real number fields $K$ by the same method, providing a broader framework for Conjecture~\ref{conj:Larsen}.
    \end{enumerate}
\end{rmk}

The main strategy of our proof is to generalize the Heegner point method from \cite{Im07}, though in a different way from the approach taken by the second author. Roughly speaking, we construct an infinite sequence of primes $\{p_n\}$ satisfying certain properties, and consider the Heegner points in the ring class fields of conductor $p_1\cdots p_n$. We then show that the traces of these Heegner points are non-torsion under the assumption that the analytic rank is at most 1, and that they generate an infinitely generated group, thereby establishing the theorem. 

In Section~\ref{preliminary_lemmas}, we give some preliminary lemmas which are needed in proving main theorems. In Section~\ref{construction_primes}, we will construct an infinite sequence of primes $\{p_n\}$ which is need in the proof of Theorem~\ref{main_theorem}. Finally, in Section~\ref{proof}, we will prove Theorem~\ref{main_theorem}.

\section{Preliminary lemmas}\label{preliminary_lemmas}

We first prove a lemma describing the Galois group structure of ring class fields of an imaginary quadratic field. We denote a cyclic group of order $n$ by $C_n$ and if $d\mid n$ for a positive integer $d$, then $C_n$ has a unique subgroup of order $d$, so we consider $C_d$ as a subgroup of $C_n$.

\begin{lemma}\label{Galois_group_structure} Let $K$ be an imaginary quadratic field. Let $H$ be the Hilbert class field of $K$ and let $H_c$ be the ring class field of $K$ with conductor $c$. Let $\{p_n\}$ be a sequence of rational primes which are inert in $K$. Then, we have the following:
    \begin{enumerate}[\normalfont (a)]
    \item  For $i\neq j$, $H_{p_i}$ and $H_{p_j}$ are linearly disjoint over $H$.
    \item For $n \geq 1$, $H_{p_1 \cdots p_n} = H_{p_1} \cdots H_{p_n}$.
    \item For $n \geq 1$, we have 
    \begin{align*}
        \Gal(H_{p_1\cdots p_n}/H) &\cong \Gal(H_{p_1}/H) \times \cdots \times \Gal(H_{p_n}/H) \\ 
        &\cong C_{p_1+1} \times \cdots \times C_{p_n+1}.
    \end{align*}
    \end{enumerate}
\end{lemma}
\begin{proof}
For (a), we note that $H_{p_i} \cap H_{p_j} = H$, since all primes of $K$ are unramified in $H_{p_i} \cap H_{p_j}$. Thus (a) follows.

For (b) and (c), we first note that \cite[Section~9]{Co22}
    \[\Gal(H_c/H) \cong (\mathcal{O}_K/c\mathcal{O}_K)^*/ (\bbz/c\bbz)^*.\]
    Therefore, we have
    \begin{align*}
    \Gal(H_{p_i}/H) \cong \frac{(\mathcal{O}_K/p_i\mathcal{O}_K)^*}{(\bbz/p_i\bbz)^*} \cong C_{p_i^2-1}/C_{p_i-1} \cong C_{p_i+1}
    \end{align*}
    and 
    \begin{align*}
        \Gal(H_{p_1\cdots p_n}/H) &\cong \frac{(\mathcal{O}_K/p_1\cdots p_n\mathcal{O}_K)^*}{(\bbz/p_1\cdots p_n\bbz)^*} \\
        &\cong \frac{(\mathcal{O}_K/p_1\mathcal{O}_K)^* \times \cdots \times (\mathcal{O}_K/p_n\mathcal{O}_K)^*}{(\bbz/p_1\bbz)^* \times \cdots \times (\bbz/p_n\bbz)^*} \\
        &\cong \frac{(\mathcal{O}_K/p_1\mathcal{O}_K)^*}{(\bbz/p_1\bbz)^*} \times \cdots \times \frac{(\mathcal{O}_K/p_n\mathcal{O}_K)^*}{(\bbz/p_n\bbz)^*} \\
        &\cong C_{p_1^2-1}/C_{p_1-1}\times \cdots C_{p_n^2-1}/C_{p_n-1}\\
        &\cong  C_{p_1+1} \times \cdots \times C_{p_n+1}.
    \end{align*}
Since $H_{p_1},\ldots,H_{p_n}$ are linearly disjoint over $H$ by (a), we complete the proof of (b)~and~(c).
\end{proof}

Next, we list some properties of Heegner points. Let $E/\bbq$ be an elliptic curve over $\bbq$ of conductor $N$ and let $K$ be an imaginary quadratic field such that the pair $(E,K)$ satisfies the Heegner hypothesis. Recall that the Heegner hypothesis is 
\begin{hypothesis}\cite[Hypothesis~3.9]{Da04}\label{Heegner_hypothesis}
    All primes dividing $N$ split in $K/\bbq$.
\end{hypothesis}

The Heegner system attached to $(E,K)$ is a collection of Heegner points $\{P_n\}_{(n,N)=1}$ indexed by integers $n$ prime to $N$, and satisfying certain conditions. See \cite[Definition~3.12]{Da04} for the detail.

Let $\{P_n\}_{(n,N)=1}$ be a fixed nontrivial Heegner system attached to $(E,K)$, which exists by \cite[Theorem~3.13]{Da04}.

\begin{lemma}\cite[Proposition~3.10]{Da04}\label{heegner_system_trace}
    Let $n$ be a positive integer and let $\ell$ be a prime such that $(n\ell, N)=1$. Suppose $\ell \nmid n$ and that $\ell$ is inert in $K$. Then, we have  
    \[\Tr_{H_{n\ell}/H_n}(P_{n\ell}) = a_\ell P_n.\]
\end{lemma}

\begin{lemma}\cite[Lemma~3.14]{Da04}\label{torsion_finiteness}
    Let $H_\infty$ be the union of all ring class fields $H_c$, where $(c,N)=1$. Then, the torsion subgroup of $E(H_\infty)$ is finite.
\end{lemma}

We denote the point $P_1$ by $P_H$. Also, for intermediate fields $H/L/K$, we let 
\[P_L := \Tr_{H/L}(P_H).\]

\begin{lemma}\cite[Theorem~3.20]{Da04}\label{gross_zagier}
    Let $\langle\:,\:\rangle$ denote the canonical N\'{e}ron-Tate height on $E(K)$ extended by linearity to a Hermitian pairing on $E(K) \otimes \bbc$. Then, we have 
    \[\langle P_K,P_K \rangle \doteq L'(E/K,1),\]
    where $\doteq$ denotes equality up to a non-zero fudge factor, which can be made explicit. 
    
    In particular, $P_K$ is non-torsion if and only if $L'(E/K,1) \neq 0$.
\end{lemma}

Finally, we need the following lemma that is used in proving the infinitude of the  rank of an elliptic curve over an infinite field extension.

\begin{lemma} \cite[Lemma~2.5]{Im07}\label{infinite_rank}
    Let $E/K$ be an elliptic curve over a number field $K$, $L$ be an (infinite) Galois extension of $K$, and let $\{P_n\}$ be an infinite sequence of points in $E(L)$. Denote by $\mathcal S$ the subgroup of $E(L)$ generated by the points $P_n$. Suppose the following conditions:
    \begin{enumerate}[\normalfont (i)]
        \item The torsion subgroup of $E(L)$ is finite.
        \item $\mathcal S$ is not finitely generated.
    \end{enumerate}
    Then, we have  that  $$\dim \mathcal S \otimes \bbq = \dim E(L) \otimes \bbq = \infty.$$
\end{lemma}

\section{Construction of infinite sequence of certain primes for Theorem~\ref{main_theorem}}\label{construction_primes}

Let $E/\bbq$ be an elliptic curve over $\bbq$ of conductor $N$. For a positive integer $m$, consider the natural representation
\begin{equation}\label{phi_m}
    \phi_m:\Gal(\bbq(E[m])/\bbq)) \rightarrow \GL_2(\bbz/m\bbz),
\end{equation}
where $\bbq(E[m])$ is an $m$-division field of $E$, i.e., the field of definition of all $m$-torsion points of $E$, recalling that $E[m] \cong \bbz/m\bbz \times \bbz/m\bbz$. By composing with the projection $G_\bbq \rightarrow \Gal(\bbq(E[m])/\bbq))$, we obtain the natural representation
\begin{equation}\label{phi_m_Gq}
    \phi_m:G_\bbq \rightarrow \GL_2(\bbz/m\bbz),
\end{equation}
which we also denote by $\phi_m$. Note that the image of $\phi_m$ given in~\eqref{phi_m} is same as that of~\eqref{phi_m_Gq}.

First, suppose $E$ does not have complex multiplication. Then by Serre's open image theorem \cite{Se72}, there exists an integer $M$ such that $\phi_m$ is surjective whenever $(m,M)=1$.

\begin{lemma}\label{quadratic_field_non_cm}
    Let $E/\bbq$ be an elliptic curve over $\bbq$ of conductor $N$ without complex multiplication. Assume that $E/\bbq$ has analytic rank at most~$1$. Let $M$ be the constant such that $\phi_m$ given in~\eqref{phi_m} is surjective whenever $(m,M)=1$. Then, there exists an imaginary quadratic field $K$ with discriminant $d_K$ satisfying the following properties:
    \begin{enumerate}[\normalfont (i)]
        \item $d_K \equiv 1 \pmod{4}$.
        \item $(d_K,NM)=1$.
        \item $(E,K)$ satisfies the Heegner hypothesis.
        \item $L'(E/K,1) \neq 0$.
    \end{enumerate}
\end{lemma}
\begin{proof}
    First, suppose the analytic rank of $E/\bbq$ is $0$, i.e., suppose $L(1) \neq 0$. Then the root number $\epsilon$ is 1. By \cite[Theorem~(i)]{BFH90}, we obtain an imaginary quadratic field $K$ such that all primes dividing $2NM$ split in $K$ and $L_D(s)$ has a simple zero at $s=1$, where $L_D = L_{d_K}$ is the quadratic twist of an $L$-function of $E$ over $K$.

   Next, suppose the analytic rank of $E/\bbq$ is $1$, so $L'(1) \neq 0$. Then the root number $\epsilon$ is $-1$, and  \cite[Theorem~(ii)]{BFH90} guarantees the existence of an imaginary quadratic field $K$ such that all primes dividing $2NM$ split in $K$ and $L_D(1) \neq 0$. 
    
    In both cases, the properties from~(i) to~(iii) are satisfied, and since $L(E/K,s) = L(s)L_D(s)$, we conclude $L'(E/K,1) \neq 0$.
\end{proof}

\begin{lemma}\label{p_lemma_non_cm} 
    Let $E/\bbq$ be an elliptic curve over $\bbq$ of conductor $N$ without complex multiplication. Assume that $E/\bbq$ has analytic rank at most~$1$. Let $M$ be an integer as in Lemma~\ref{quadratic_field_non_cm} and let $K$ be the imaginary quadratic field satisfying the properties from {\normalfont (i)} to {\normalfont (iv)} given in Lemma~\ref{quadratic_field_non_cm}. Fix a prime $q$ satisfying $(q,2d_KNM)=1$. Then, there exist infinitely many primes~$p$ satisfying the following properties:
    \begin{enumerate}[\normalfont (i)]
        \item $p \equiv -1 \pmod {q}$,
        \item $p$ is inert in $K$,
        \item $E$ has a good reduction at $p$, and
        \item $q \nmid a_p$ where $a_p = 1+p-\lvert E(\mathbb{F}_p) \rvert$.
    \end{enumerate}
\end{lemma}
\begin{proof}
    The idea of the proof is inspired by \cite[Section~2]{Co04}. The property~(ii) is equivalent to $\left(\frac{d_K}{p}\right)=-1$, which, in turn, is equivalent to $p \equiv a \pmod{\lvert d_K \rvert}$ for some set of values~$a$ modulo $\lvert d_K \rvert$. We choose one such $a$ so that whenever $p \equiv a \pmod{\lvert d_K \rvert}$, the prime $p$ remains  inert in $K$. By the Chinese Remainder Theorem, there exists an integer $b$ such that if $p \equiv b \pmod{q\lvert d_K \rvert}$, then $p \equiv -1 \pmod {q}$ and $p \equiv a \pmod{\lvert d_K \rvert}$. Setting $m=q\lvert d_K \rvert$, we conclude that $p \equiv b \pmod{m}$ satisfies both properties~(i) and (ii).
    
    Now, consider the representation given in \eqref{phi_m},
    \[\phi_m:\Gal(\bbq(E[m])/\bbq)) \rightarrow \GL_2(\bbz/m\bbz).\]
    Then, we have:
    \[\Tr(\phi_m(\sigma_p)) \equiv a_p \pmod{m}, \quad\text{ and }\quad \det(\phi_m(\sigma_p)) \equiv p \pmod{m},\]
    where $\sigma_p$ is a Frobenius map at $p$ in $\Gal(\bbq(E[m])/\bbq))$. Thus, it suffices to find infinitely many primes $p$ satisfying 
    \[\Tr(\phi_m(\sigma_p)) \not\equiv 0 \pmod{q}, \quad\text{ and }\quad \det(\phi_m(\sigma_p)) \equiv b \pmod{m}.\]
    By the classical Chebotarev density theorem, it suffices to show that the set 
    \[\{g \in \Gal(\bbq(E[m])/\bbq))\:|\:\Tr(\phi_m(g)) \not\equiv 0 \pmod{q},\:\det(\phi_m(g)) \equiv b \pmod{m}\}\]
    is non-empty. However, this is clear by Serre's open image theorem~\cite{Se72}, since $(m,M)=1$ due to our choices of $K$ and $q$.
\end{proof}

Next, suppose $E/\bbq$ has complex multiplication by an order $\mathcal{O}$ of the imaginary quadratic field $F$. Then $E[m]$ is a free $\mathcal{O}/m\mathcal{O}$-module of rank 1, so the restriction of $\phi_m$ given in \eqref{phi_m_Gq} to $G_F$ factors through
\[\left.{\phi_m}\right|_{G_F} : G_F \rightarrow \GL_1(\mathcal{O}/m\mathcal{O}) \rightarrow \GL_2(\bbz/m\bbz).\]
Let $I_m = \phi_m(G_F)$ be the image of $G_F$ under $\phi_m$. 

Consider the case when $m=p$ where $p$ is a prime relatively prime to the discriminant $d_F$ of~$F$, and $E$ has a good reduction at $p$. Then, by \cite[Lemma~7.1]{Zy15}, the map 
\[G_F \rightarrow \GL_1(\mathcal{O}/p\mathcal{O})\] 
is surjective. Therefore, $I_p$ is equal to the image of 
\[\GL_1(\mathcal{O}/p\mathcal{O}) \rightarrow \GL_2(\bbz/p\bbz).\]

Note that 
\[\GL_1(\mathcal{O}/p\mathcal{O}) \cong (\mathcal{O}/p\mathcal{O})^* \cong (\mathbb{F}_p \times \mathbb{F}_p)^* \cong \mathbb{F}_p^* \times \mathbb{F}_p^*\]
if $p$ splits in $F$ and 
\[\GL_1(\mathcal{O}/p\mathcal{O}) \cong (\mathcal{O}/p\mathcal{O})^* \cong \mathbb{F}_{p^2}^*\]
if $p$ is inert in $F$. 

If $p$ splits in $F$, then the image of $\GL_1(\mathcal{O}/p\mathcal{O})$ in $\GL_2(\bbz/p\bbz)$ is isomorphic to $\mathbb{F}_p^* \times \mathbb{F}_p^*$. Hence, the image is a split Cartan subgroup of $\GL_2(\bbz/p\bbz)$ and is conjugate to the diagonal subgroup of $\GL_2(\bbz/p\bbz)$. In particular, the composition 
\begin{equation}\label{determinant_GL1}
    \GL_1(\mathcal{O}/p\mathcal{O}) \rightarrow \GL_2(\bbz/p\bbz) \xrightarrow{\det} (\bbz/p\bbz)^*
\end{equation}
is surjective.

If $p$ is inert in $F$, then the image of $\GL_1(\mathcal{O}/p\mathcal{O})$ in $\GL_2(\bbz/p\bbz)$ is isomorphic to~$\mathbb{F}_{p^2}^*$. Hence, the image is a non-split Cartan subgroup of $\GL_2(\bbz/p\bbz)$. Note that in this case, the composition~\eqref{determinant_GL1} is just the norm map 
$\mathbb{F}_{p^2}^* \rightarrow \mathbb{F}_p^*$, 
which is surjective.

In either case, we conclude that the composition map 
\[G_F \rightarrow \GL_1(\mathcal{O}/p\mathcal{O}) \rightarrow \GL_2(\bbz/p\bbz) \xrightarrow{\det} (\bbz/p\bbz)^*\]
is surjective, and the image $I_p$ has order $\geq (p-1)^2 = \phi(p)^2$. Thus, for any every $b \in (\bbz/p\bbz)^*$, there exist at least $p-1 = \phi(p)$ elements in $I_p$ with determinant $b$.

Now consider the general case for $m$. By Chinese Remainder Theorem, if $m$ is a square-free positive integer such that $(m,d_FN)=1$, then the composition map 
\[G_F \rightarrow \GL_1(\mathcal{O}/m\mathcal{O}) \rightarrow \GL_2(\bbz/m\bbz) \xrightarrow{\det} (\bbz/m\bbz)^*\]
is surjective and the image $I_m$ has order $\geq \phi(m)^2$. Thus, for every $b \in (\bbz/m\bbz)^*$, there are at least $\phi(m)$ elements in $I_m$ with determinant $b$. We record this as a lemma. 

\begin{lemma}\label{I_m}
    Let $E/\bbq$ be an elliptic curve of conductor $N$. Assume that $E$ have complex multiplication by an order $\mathcal{O}$ of the imaginary quadratic field $F$. Let $m$ be a square-free positive integer such that $(m,d_FN)=1$, and let $I_m:= \phi_m(G_F)$ where $\phi_m$ is given in \eqref{phi_m_Gq}. Then for every $b \in (\bbz/m\bbz)^*$, there are at least $\phi(m)$ elements in $I_m$ with determinant $b$.
\end{lemma}

\begin{lemma}\label{quadratic_field_cm}
    Let $E/\bbq$ be an elliptic curve of conductor $N$. Assume that $E/\bbq$ has analytic rank at most~$1$ and that $E$ have complex multiplication by an order $\mathcal{O}$ of the imaginary quadratic field $F$. There exists an imaginary quadratic field $K$ satisfying the following properties:
    \begin{enumerate}[\normalfont (i)]
        \item $d_K \equiv 1 \pmod{4}$.
        \item $(d_K,d_FN)=1$.
        \item $(E,K)$ satisfies the Heegner hypothesis.
        \item $L'(E/K,1) \neq 0$.
    \end{enumerate}
\end{lemma}
\begin{proof}
    The proof is identical to that of Lemma~\ref{quadratic_field_non_cm}.
\end{proof}

\begin{lemma}\label{p_lemma_cm}
     Let $E/\bbq$ be an elliptic curve of conductor $N$. Assume that $E/\bbq$ has analytic rank at most~$1$ and that $E$ have complex multiplication by an order $\mathcal{O}$ of the imaginary quadratic field $F$. Let $K$ be the imaginary quadratic field satisfying the properties (i) to (iv) in Lemma~\ref{quadratic_field_cm}. Fix a prime $q$ satisfying $(q,2d_Kd_FN)=1$, $(d_F/q)=1$, and $q>1+2\lvert d_K\rvert^4/\phi(\lvert d_K\rvert)$. Then, there exist infinitely many primes~$p$ satisfying the following properties:
    \begin{enumerate}[\normalfont (i)]
        \item $p \equiv -1 \pmod {q}$,
        \item $p$ is inert in $K$,
        \item $E$ has a good reduction at $p$, and
        \item $q \nmid a_p$ where $a_p = 1+p-\lvert E(\mathbb{F}_p) \rvert$.
    \end{enumerate}
\end{lemma}
\begin{proof}
    The condition (ii) is equivalent to $\left(\frac{d_K}{p}\right)=-1$, which, in turn, is equivalent to $p \equiv a \pmod{\lvert d_K \rvert}$ for some set of values $a$ modulo $\lvert d_K \rvert$. We choose one such $a$ so that whenever $p \equiv a \pmod{\lvert d_K \rvert}$, the prime $p$ remains  inert in $K$. By the Chinese Remainder Theorem, there exists an integer $b$ such that if $p \equiv b \pmod{q\lvert d_K \rvert}$, then $p \equiv -1 \pmod {q}$ and $p \equiv a \pmod{\lvert d_K \rvert}$. Setting $m=q\lvert d_K \rvert$, we conclude that $p \equiv b \pmod{m}$ satisfies both conditions (i) and (ii).
    
    Now, consider the representation given in \eqref{phi_m},
    \[\phi_m:\Gal(\bbq(E[m])/\bbq)) \rightarrow \GL_2(\bbz/m\bbz).\]
    Then, we have:
    \[\Tr(\phi_m(\sigma_p)) \equiv a_p \pmod{m}, \quad\text{ and }\quad \det(\phi_m(\sigma_p)) \equiv p \pmod{m},\]
    where $\sigma_p$ is a Frobenius at $p$ in $\Gal(\bbq(E[m])/\bbq))$. Thus it suffices to find infinitely many primes $p$ satisfying 
    \[\Tr(\phi_m(\sigma_p)) \not\equiv 0 \pmod{q},\quad \det(\phi_m(\sigma_p)) \equiv b \pmod{m}.\]
    By the Chebotarev density theorem, it suffices to show that the set 
    \[\{g \in \Gal(\bbq(E[m])/\bbq))\:|\:\Tr(\phi_m(g)) \not\equiv 0 \pmod{q},\:\det(\phi_m(g)) \equiv b \pmod{m}\}\]
    is non-empty. Thus, the problem reduces to proving that the image of $\phi_m$ contains an element $A$ satisfying 
    \[\Tr(A) \not\equiv 0 \pmod{q},\quad \text{ and }\quad \det(A) \equiv b \pmod{m}.\]
    We will find this element $A$ in the image $I_m := \phi_m(G_F)$.

    Assume that $A\in I_m$ satisfies  $\det(A) \equiv b \pmod{m}$ and $\Tr(A) \equiv 0 \pmod{q}$. Since reducing modulo~$q$ transforms $I_m$ into $I_q$ and $b$ into $-1$,  we obtain an element $A$ mod $q$ in $I_q$ satisfying $\det(A) \equiv -1 \pmod{q}$ and $\Tr(A) \equiv 0 \pmod{q}$. Given that $(d_F/q)=1$, the image $I_q$ is conjugate to the diagonal subgroup of $\GL_2(\bbz/q\bbz)$. By an appropriate change of basis, we may assume that $I_q$ itself is the diagonal subgroup of $\GL_2(\bbz/q\bbz)$. Consequently, $A$ mod $q$ must take one of the following forms:
    \begin{equation*}
        \begin{bmatrix}
            1 & 0 \\
            0 & -1
        \end{bmatrix}
        \quad \text{or} \quad 
        \begin{bmatrix}
            -1 & 0 \\
            0 & 1
        \end{bmatrix}.
    \end{equation*}
   This means that that the number of elements $A$ in $I_m$ satisfying $\det(A) \equiv b \pmod{m}$ and $\Tr(A) \equiv 0 \pmod{q}$ is at most $2\lvert d_K\rvert^4$.
   
   On the other hand, by Lemma~\ref{I_m}, the number of $A$ in $I_m$ satisfying $\det(A) \equiv b \pmod{m}$ is at least $ \phi(m) = (q-1)\phi(\lvert d_K\rvert)$. By our choice of $q$ such that $q>1+2\lvert d_K\rvert^4/\phi(\lvert d_K\rvert)$, we conclude that there must exist an element  $A\in I_m$ such that  $\det(A) \equiv b \pmod{m}$ and $\Tr(A) \not\equiv 0 \pmod{q}$. This completes the proof.
\end{proof}

\section{Proof of Theorem~\ref{main_theorem}}\label{proof}

In this section, we prove Theorem~\ref{main_theorem}. We begin by fixing an elliptic curve $E/\bbq$ over~$\bbq$ with analytic rank at most~$ 1$, and a topologically finitely generated subgroup $G$ of $G_\bbq$. Let $N$ denote  the conductor of $E$. 

If $E$ does not have complex multiplication, then we fix an imaginary quadratic field $K$ satisfying the properties (i)--(iv) given in Lemma~\ref{quadratic_field_non_cm}, and a prime $q$ given in Lemma~\ref{p_lemma_non_cm}, as well as  an infinite sequence of distinct primes $\{p_n\}$ satisfying the properties (i)--(iv)  in Lemma~\ref{p_lemma_non_cm}. 
If $E$ has complex multiplication, then we fix an imaginary quadratic field $K$ satisfying the properties (i)--(iv) given in Lemma~\ref{quadratic_field_cm} and  a prime $q$ given in Lemma~\ref{p_lemma_cm}, as well as an infinite sequence of distinct primes $\{p_n\}$ satisfying the properties (i)--(iv) in Lemma~\ref{p_lemma_cm}.

Let $H$ be the Hilbert class field of $K$, and let $H_c$ be the ring class field of $K$ of conductor~$c$. By \cite[Theorem~3.13]{Da04}, there exists a non-trivial Heegner system $\{P_n\}_{(n,N)=1}$ attached to $(E,K)$. 

We prove the theorem analyzing two cases: when $G$ fixes $K$, and when $G$ does not fix $K$.

\subsection{When $G$ fixes $K$}\label{g_fixes_k}

If $G$ fixes $K$, then $G$ is contained in $\Gal(\overline{\bbq}/K)$. Let $G_0 := G \cap \Gal(\overline{\bbq}/H)$. Then $G_0$ is an open normal subgroup of $G$ of finite index, and $G/G_0$ embeds into $\Gal(H/K)$. As an open subgroup of topologically finitely generated profinite group, $G_0$ is topologically finitely generated. Let $G_0$ be topologically generated by $r$ elements. 

For each integer $n \geq r$, let $\tilde{H_n}$ be the fixed field of 
\[C_{\frac{p_1+1}{q}} \times \cdots \times C_{\frac{p_n+1}{q}} \subseteq C_{p_1+1} \times \cdots \times C_{p_n+1}.\]
By Lemma~\ref{Galois_group_structure}, we identify the Galois group
\[\Gal(H_{p_1\cdots p_n}/H) \cong C_{p_1+1} \times \cdots \times C_{p_n+1}.\]
Then $\tilde{H_n}$ is Galois over $H$, with Galois group 
\[\Gal(\tilde{H_n}/H) \cong C_q \times \cdots \times C_q  = C_q^n.\]

Let $H_n^* = (\tilde{H_n})^{G_0}$ be the fixed field of the image of $G_0$ in $\Gal(\tilde{H_n}/H)$. Since $G_0$ is topologically generated by $r$ elements, its image in $\Gal(\tilde{H_n}/H)$ is also generated by $r$ elements. By Lemma~\ref{index_non_cm} given below, the index of the image of $G_0$ in $\Gal(\tilde{H_n}/H)$ is at least $ q^{n-r}$. Consequently, we obtain $$[H_n^*:H] \geq q^{n-r}.$$

\begin{lemma}\label{index_non_cm}
    Let $q$ be a prime and $G \leq C_q^n$ be a subgroup. Let $G$ be generated by $r \leq n$ elements. Then $[C_q^n:G] \geq q^{n-r}$.
\end{lemma}
\begin{proof}
    Let $m$ be the minimal number of generators of $G$. Then $G$ is isomorphic to $C_q^m$. Thus, we have  $\lvert G \rvert = q^m \leq q^r$ and $[C_q^n:G] \geq q^{n-r}$.
\end{proof}

In summary, we have towers of fields $\{\tilde{H_n}\}_{n\geq r}$ and $\{H_n^*\}_{n \geq r}$ satisfying $$H_n^* \subseteq \tilde{H_n}, \quad [\tilde{H_n}:H] = q^n, \quad \text{ and } \quad [H_n^*:H] \geq q^{n-r}.$$

Let $\{P_n\}_{(n,N)=1}$ be a fixed non-trivial Heegner system. For each integer $n \geq r$, let $P_n$ be the Heegner point $P_{p_1\ldots p_n}$, and let 
\[P_n^* = \Tr_{H_{p_1 \cdots p_{n}}/H_n^*}(P_n).\]
Then, $P_n^*$ are fixed by $G_0$.

Recall that $G/G_0$ embeds into $\Gal(H/K)$. Then, its image is of the form $\Gal(H/L)$ for some intermediate field $L$ between $K$ and $H$. This means that representatives $\sigma_1,\ldots,\sigma_s$ of $G/G_0$ consist of the extensions of elements of $\Gal(H/L)$. For each integer $n \geq r$, let 
\[\Tr_{G/G_0}(P_n^*) := \sum_{i=1}^s \sigma_i(P_n^*).\]
Then, all $\Tr_{G/G_0}(P_n^*)$ are fixed by $G$.

 Suppose $\{\Tr_{G/G_0}(P_i^*)\}_{i \geq r}$ generate a finitely generated group. Then there exists some integer $m$ such that all $\Tr_{G/G_0}(P_i^*)$ are defined over $H_m^*$. Consequently,  their further trance $\Tr_{H_m^*/H}(\Tr_{G/G_0}(P_i^*))$ are all defined over $H$. If $n \geq r$, then  by Lemma~\ref{heegner_system_trace}, we obtain
\[\Tr_{H_{p_1\cdots p_n}/H}(P_n) = a_{p_1}\cdots a_{p_n}P_H.\]
Moreover, we have
\begin{align*}
    \Tr_{H_{p_1\cdots p_n}/H}(P_n) = \Tr_{H_n^*/H}(P_n^*).
\end{align*}
Thus, it follows that
\[\Tr_{H_n^*/H}(P_n^*) = a_{p_1}\cdots a_{p_n}P_H,\quad \text{ for } n \geq r.\]
By taking the trace $\Tr_{G/G_0}$ and recalling that $\Gal(H_n^*/K)$ is abelian, we obtain
\begin{equation}\label{trace_relation}
    \Tr_{H_n^*/H}(\Tr_{G/G_0}(P_n^*)) = a_{p_1} \cdots a_{p_n}\Tr_{H/L}(P_H).
\end{equation}
The right-hand side of \eqref{trace_relation} is just $a_{p_1} \cdots a_{p_n}P_L$. On the other hand, if $n \geq m$, then the left-hand side of \eqref{trace_relation} satisfies that 
\begin{align*}
    \Tr_{H_n^*/H}(\Tr_{G/G_0}(P_n^*)) &= \Tr_{H_m^*/H}\Tr_{H_n^*/H_m^*}(\Tr_{G/G_0}(P_n^*)) \\
    &= [H_n^*:H_m^*]\Tr_{H_m^*/H}(\Tr_{G/G_0}(P_n^*)).
\end{align*}
Thus, we obtain
\begin{equation}\label{relation}
    [H_n^*:H_m^*]\Tr_{H_m^*/H}(\Tr_{G/G_0}(P_n^*)) = a_{p_1} \cdots a_{p_n}P_L,\quad \text{ for } n \geq m.
\end{equation}

By the Mordell-Weil theorem, $E(H)$ is finitely generated, say by $Q_1,\ldots,Q_k$ modulo torsion. Let 
\[\Tr_{H_m^*/H}(\Tr_{G/G_0}(P_i^*)) \equiv \sum_{j=1}^k c_{i,j}Q_j\text{ modulo torsion},\quad \text{ for } i > r, \]
and 
\[P_L \equiv \sum_{j=1}^k c_{j}Q_j\text{ modulo torsion},\]
where $c_{i,j}$ and $c_j$ are all integers. By Lemma~\ref{gross_zagier} and Lemma~\ref{quadratic_field_non_cm}, $P_K$ is non-torsion. Since 
\[P_K = \Tr_{L/K}(P_L),\]
it follows that $P_L$ is also non-torsion, implying that  some $c_{j}$ is nonzero. From the relation~\eqref{relation}, we obtain
\[[H_n^*:H_m^*]c_{n,j} = a_{p_1}\cdots a_{p_n}c_{j}.\]
Note that $[H_n^*:H_m^*] =q^N$ for some  $N \geq {n-m-r}$ when $n$ is sufficiently large, and  by Lemma~\ref{p_lemma_non_cm}, $q \nmid a_{p_i}$ for any $i$. This implies that $q^{n-m-r} \mid c_j$ for all sufficiently large $n$, which is impossible. 
Therefore, we conclude that $\{\Tr_{G/G_0}(P_i^*)\}_{i \geq r}$ generate a group that is not finitely generated. 

Let $H_\infty$ denote the union of all ring class fields $H_c$, where $(c,N)=1$, and let $L$ be the maximal Galois extension of $\bbq$ contained in $H_\infty^G$. By applying Lemma~\ref{infinite_rank} together with Lemma~\ref{torsion_finiteness}, this completes the proof of Theorem~\ref{main_theorem} when $G$ fixes $K$.

\

Next, we prove the theorem when $G$ does not fix $K$.

\subsection{When $G$ does not fix $K$}\label{g_does_not_fix_k} Now we remove the assumption that $G$ fixes $K$. The approach of  this subsection is inspired by~\cite{Im07}.

Let $G_1 = G \cap \Gal(\overline{\bbq}/K)$ and choose $\sigma \in G - G_1$. Note that $G_1$ is an index-two subgroup of $G$. Moreover, if we consider the images of the groups $G$ and $G_1$ under the projection $\Gal(\overline{\bbq}/\bbq) \rightarrow \Gal(H_n/\bbq)$, we have the decomposition $G = G_1 \rtimes \bbz/2\bbz$, where $\sigma$ acts as an involution. 

Now, we prove that the rank of $E((\tilde{H_n})^{G})$ over the fixed subfield under $G$ is unbounded. Suppose, for the sake of contradiction, it is bounded. Then there exists a positive integer $m$ such that the rank of $E((\tilde{H_n})^{G})$ is equal to the rank of $E((\tilde{H_m})^{G})$ whenever $n \geq m$. Therefore, $\sigma$ acts as an involution on the vector space 
\[M_n := (E((\tilde{H_n})^{G_1}) \otimes \bbq)/(E((\tilde{H_m})^{G_1}) \otimes \bbq),\quad \text{ for } n \geq m\] 
 with fixed vector space 0. Therefore, $\sigma$ acts on $M_n$ by  $-1$.

Let 
\[\rho : \Gal((\tilde{H_n})^{G_1}/\bbq) \rightarrow \Aut(M_n)\]
be the representation. Since $\Gal(H/K)$ is finite, by taking sufficiently large $m$, we may assume that $[(\tilde{H_n})^{G_1}:(\tilde{H_m})^{G_1}]$ is a power of $q$. Thus, $\Gal((\tilde{H_n})^{G_1}/(\tilde{H_m})^{G_1})$ is isomorphic to a product of cyclic $q$-groups. For each generator $\tau$ of each of these cyclic $q$-groups, $\tau^2$ acts on $M_n$ trivially, i.e., by identity, since 
\[\rho(\tau^2) = (-1)\rho(\tau)(-1)\rho(\tau) = \rho(\sigma\tau\sigma^{-1}\tau) = \rho(\tau^{-1}\tau) = \rho(1)=1.\]
Here, by abuse of notation, we denote by $\sigma$  its restriction to $(\tilde{H_n})^{G_1}$. Therefore, $\tau^2$ acts trivially on $M_n$, and hence $\tau$ acts on $M_n$ trivially, since the order of $\tau$ is a power of $q$. It follows that $\Gal((\tilde{H_n})^{G_1}/(\tilde{H_m})^{G_1})$ acts trivially on $M_n$, implying $M_n=0$. This contradicts our eariler result in Subsection~\ref{g_fixes_k}, which established that the rank of $E((\tilde{H_n})^{G_1})$ is unbounded for large $n$. Therefore, the rank of $E((\tilde{H_n})^{G})$ over the fixed subfield under $G$ is also unbounded, which completes the proof of Theorem~\ref{main_theorem}.



\begin{thebibliography}{99}

\bibitem{BFH90} D. Bump, S. Friedberg, J. Hoffstein, Nonvanishing theorems for L-functions of modular forms and their derivatives. Invent. Math. 102 (1990), no. 3, 543–618.

\bibitem{BI08} F. Breuer, B.-H. Im, Heegner points and the rank of elliptic curves over large extensions of global fields. Canad. J. Math. 60 (2008), no. 3, 481–490.

\bibitem{Co04} A. C. Cojocaru, Questions about the reductions modulo primes of an elliptic curve. Number theory, 61–79. CRM Proc. Lecture Notes, 36, American Mathematical Society, Providence, RI, 2004.

\bibitem{Co22} D. A. Cox, Primes of the form x2+ny2—Fermat, class field theory, and complex multiplication. Third edition, AMS Chelsea Publishing, Providence, RI, [2022], ©2022. xv+533 pp.

\bibitem{Da04} H. Darmon, Rational points on modular elliptic curves. CBMS Reg. Conf. Ser. Math., 101
Published for the Conference Board of the Mathematical Sciences, Washington, DC; by the American Mathematical Society, Providence, RI, 2004.

\bibitem{DD09} T. Dokchitser, V. Dokchitser, A note on Larsen's conjecture and ranks of elliptic curves. Bull. Lond. Math. Soc. 41 (2009), no. 6, 1002–1008.

\bibitem{FJ74} G. Frey, M. Jarden, Approximation theory and the rank of abelian varieties over large algebraic fields. Proc. London Math. Soc. (3) 28 (1974), 112–128.

\bibitem{Im06} B.-H. Im, The rank of elliptic curves with rational 2-torsion points over large fields, Proc. Amer. Math. Soc. 134 (2006), no. 6, 1623–1630.

\bibitem{Im06-2} B.-H. Im, Mordell-Weil groups and the rank of elliptic-curves over large fields. Canad. J. Math. 58 (2006), no. 4, 796–819.

\bibitem{Im07} B.-H. Im, Heegner points and Mordell-Weil groups of elliptic curves over large fields, Trans. Amer. Math. Soc. 359 (2007), no.12, 6143–6154.

\bibitem{IL08} B.-H. Im, M. Larsen, Abelian varieties over cyclic fields, Amer. J. Math. 130 (2008), no.5, 1195–1210.

\bibitem{IL13} B.-H. Im, M. Larsen, Some applications of the Hales-Jewett theorem to field arithmetic. Israel J. Math. 198 (2013), no. 1, 35–47.

\bibitem{IL19} B.-H. Im, M. Larsen, Abelian varieties and finitely generated Galois groups, Abelian varieties and number theory, 1–12. Contemp. Math., 767 American Mathematical Society, [Providence], RI, [2021], ©2021

\bibitem{La03} M. Larsen, Rank of elliptic curves over almost separably closed fields. Bull. London Math. Soc. 35 (2003), no. 6, 817–820.

\bibitem{Se72} J. P. Serre, Propri´et´es galoisiennes des points d’ordre fini des courbes elliptiques, Invent. Math. 15 (1972), no. 4, 259–331.

\bibitem{Zy15} D. Zywina, On the possible images of the mod $\ell$ representations associated to elliptic curves over~$\bbq$, preprint at https://arxiv.org/abs/1508.07660

\end{thebibliography}
\end{document}